\documentclass[10pt,reqno,letter]{amsart}
\usepackage[english]{babel} % Palabras por defecto en Inglés, como fechas, títulos 
\usepackage{anysize}
\marginsize{4cm}{4cm}{3cm}{3cm} % Márgenes izquierda, derecha,arriba,abajo

\usepackage{xcolor,amsthm,amsmath,amssymb}%escribir en colores 

\usepackage{tikz}
\usepackage{dsfont}
\usetikzlibrary{chains,arrows.meta,quotes,bending}
\usetikzlibrary{automata, positioning}
\usepackage{capt-of}%Borra el caption

\usepackage[colorlinks,citecolor=red,urlcolor=blue]{hyperref} %Pequeño cuadro para la referencia

\numberwithin{equation}{section}%numerar las ecuaciones de acuerdo a seccion.

\newtheorem{theorem}{Theorem}[section]
\newtheorem{lemma}[theorem]{Lemma}
\newtheorem{proposition}[theorem]{Proposition}
\newtheorem{corollary}[theorem]{Corollary}
\theoremstyle{definition}

\theoremstyle{remark}

\newtheorem{remark}[theorem]{\bf Remark}

\newcommand{\C}{{\bf C}}

\newcommand{\NN}{\mathbb{N}}

\usepackage{hyperref}

\newcommand{\be}{\begin{equation}}
\newcommand{\ee}{\end{equation}}
\newcommand{\ba}{\begin{aligned}}
\newcommand{\ea}{\end{aligned}}
\newcommand{\bpm}{\begin{pmatrix}}
\newcommand{\epm}{\end{pmatrix}}
\newcommand{\q}{\quad}

\DeclareMathAlphabet{\mathdutchcal}{U}{dutchcal}{m}{n}

\begin{document}
\title[PDE Collatz]{A note on two Collatz  evolution flows}

\author[Alegría]{Francisco Alegría}
\address{Instituto de Ciencias F\'isicas y Matem\'aticas, Facultad de Ciencias, Universidad Austral de Chile, Valdivia, Chile.}
\email{franciscoalegria@uach.cl}
\thanks{F. A.'s work was partly funded by Chilean research grants Fondecyt Exploraci\'on 13220060 and FONDECYT 11251554.}

\author[Morales]{Mat\'ias Morales}
\address{Universidad del País Vasco/Euskal Herriko Unibertsitatea, UPV/EHU, Aptdo. 644, 48080, Bilbao, Spain}
\email{matiasbenjamin.morales@ehu.eus}
\thanks{M. M.'s work was partially funded by Fondecyt Exploraci\'on 13220060.} 

\author[Mu\~noz]{Claudio Mu\~noz}
\address{CNRS and Departamento de Ingenier\'{\i}a Matem\'atica and Centro
de Modelamiento Matem\'atico (UMI 2807 CNRS), Universidad de Chile, Casilla
170 Correo 3, Santiago, Chile.}
\email{cmunoz@dim.uchile.cl}
\thanks{C. M.'s work was partly funded by Chilean research grants Fondecyt Exploraci\'on 13220060, FONDECYT 1231250 and Centro de Modelamiento Matemático (CMM) BASAL fund FB210005 for center of excellence from ANID-Chile.}

\author[Poblete]{Felipe Poblete}
\address{Instituto de Ciencias F\'isicas y Matem\'aticas, Facultad de Ciencias, Universidad Austral de Chile, Valdivia, Chile.}
\email{felipe.poblete@uach.cl}
\thanks{F.P.'s work is partially supported by ANID Exploration project 13220060, and ANID project FONDECYT 1221076}

\date{\today}

\maketitle

\begin{abstract}
Two evolution models based on the generalized Collatz operator are introduced. These models are characterized by coefficients $\alpha$ and $\beta$ in the Collatz dynamics, and are suitably defined. Here, $\alpha=\beta=1$, and $\alpha=3$, $\beta=1$ correspond to the Nollatz and classical Collatz operators, respectively. In general, the first evolution model is a continuum, Fourier side based, motivated by the Cubic Szeg\H{o} operator of G\'erard and Grellier. The second evolution considers discrete time derivatives of the Collatz orbits. In this paper we describe the evolution of both models, with particular emphasis on dynamical properties. For the first one, it is proved local and global existence in the space $L^2(\mathbb T)$, and a one-to-one characterization of the existence of nontrivial periodic and unbounded orbits of the Collatz mapping in terms of particular set of solutions of this continuous Collatz flow. For the discrete part, a sort of discrete energy is introduced. This energy has the property of being conserved by the discrete flow. An estimate of each term in this energy is given, proving suitable growth bounds. Finally, the meaning of the discrete time derivative for the generalized Collatz orbits is discussed. It is proved that, except for the Nollatz and Collatz operators, the sum of coefficients related to this discrete time derivative is an increasing sequence in $n$ as the iteration parameter $n$ evolves.
\end{abstract}

\section{Introduction}

\subsection{Setting} 

Let $\mathbb N$ be the set of nonnegative integers. The Collatz problem is a question first proposed by Lothar Collatz in 1937 stated as follows. Let $n\in \mathbb{N}$ be a positive integer. Let $C(n)$ be  $n/2$ if $n$ is even, $3n+1$ if $n$ is odd. The Collatz Conjecture states that for any positive integer $n$, the process $C^{k}(n):= C\circ C \cdots C (n)$ ($k$ times) will eventually lead, as $k$ becomes larger, to the number 1 through the trivial cycle $1\longrightarrow 4\longrightarrow 2\longrightarrow1$. 

Despite its simplicity, the problem has resisted proof or disproof for over 80 years. See Lagarias \cite{Lagarias} for a comprehensive introduction to the subject. Kontorovich and Sinai \cite{KS02} showed that the orbit of the Collatz sequence follows, in some sense, a probabilistic Brownian-type motion. Some particular issues appear in the case of failure of the Collatz conjecture. One of the most interesting is the possible existence of large cycles away from the previously mentioned work \cite{Matthews}. In fact, Eliahou established a significant lower bound for the length of any nontrivial cycle \cite{Eliahou}. In any case, the Collatz conjecture remains an active area of research, see e.g. \cite{KL03,T17}. Numerical computations have verified the conjecture up to $2^{71}$ \cite{Bar20,Bar25}, and interesting extensions to larger numerical sets have been studied; these works allow the use of tools from the study of iterating continuous maps \cite{Chamberland}. The conjecture has been shown to have connections to other areas of mathematics, such as number theory and dynamical systems \cite{Crandall,Lagarias,LagariasBook}. See also \cite{Wirsching} for further details.

Very recently, T. Tao \cite{Tao:2022} improved and extended methods essentially related to the probabilistic global well-posedness of supercritical dispersive models to establish that under a suitable probabilistic law, almost all integers have a bounded Collatz orbit. This method already appears in Bourgain \cite{B94}, and has been extensively improved and used in the past years, see e.g. \cite{BT08,GGM25,T10}.

As we have done in a recent paper \cite{CMPS}, the purpose of this work is to study two Collatz motivated flow dynamics, seeking interesting connections with other PDE methods. One will be of continuum type, strongly acting on the Fourier side (e.g. as the Cubic Szeg\H{o} operator does in \cite{GG10}), and the second one will be of discrete type, inspired in certain applications in numerical analysis. For this, we first need some definitions. Let us introduce the generalized (accelerated) Collatz mapping $C_{\alpha,\beta}$, $\alpha,\beta\in 2\mathbb N +1$, $\hbox{gcd}(\alpha,\beta)=1$, $n\in\mathbb N$,
\be\label{C_alpha}
C_{\alpha,\beta}(n) = 
\begin{cases}
\dfrac{n}{2} & \text{if }n \equiv 0 \pmod 2 \vspace{0.1cm} \\
\dfrac{\alpha n+\beta}{2} & \text{if }n \equiv 1 \pmod 2.
\end{cases}
\ee
Notice that \eqref{C_alpha} is usually referred as the ``accelerated'' version of Collatz mapping and has been extensively studied in the literature, see e.g. \cite{Conway,Fractran,GGM25}. If no confusion arises, we set $C(n)=C_{\alpha,\beta}(n)$. We extend $C$ to the integers $\mathbb Z$  by odd symmetry:
\begin{equation}\label{simetriaC}
C(-n)=-C(n), \quad C(0)=0.
\end{equation} 
Let $L^2:=L^2(\mathbb T)$, with $\mathbb T=[-\pi, \pi]$ the standard torus with periodic boundary conditions. Let $u_0\in L^2$. We know that
\begin{equation}\label{fouriercondini}
u_0(x)=\sum_{n\in \mathbb{Z}}\hat u_{0,n}e^{i  n x},
\end{equation}
where $\hat u_{0,n}$ is the Fourier coefficient, given by $\hat u_{0,n}=\frac{1}{2\pi}\int_{-\pi}^{\pi}u_0(x)e^{ i nx}dx$. Then $L^2$ is a Hilbert space endowed with the inner product and norm
\[
\langle u,v\rangle = \sum_{n\in\mathbb Z} \hat u_{n} \overline{\hat v_{n}}, \quad \| u\| ^2  :=   \langle u, u\rangle.
\]
The Collatz operator acting on $u_0 \in L^2$ as in \eqref{fouriercondini} is defined as 
\begin{equation}\label{C}
  {\bf C}(u_0)(x)=\sum_{n\in \mathbb{Z}}\hat u_{0,C(n)}e^{ i n x}.
\end{equation}
Notice that this operator differs from others already defined in the literature \cite{Mori24}. As previously mentioned, our main motivation comes from the fact that \eqref{C} captures the exact action of the Collatz mapping on the Fourier side, namely, on the Fourier coefficients, as it is usually done in several dispersive models.

\subsection{Continuous Collatz PDE flow} 
In this paper we establish a new connection between the Collatz operator acting on the Hilbert space $L^2$ and the original Collatz mapping. This action is done through the nonlocal PDE in the periodic setting 
\begin{equation}\label{EDP}
\partial_t u= \C (u), \quad u(t=0)=u_0.
\end{equation}
Via \eqref{C}, this problem reduces to an infinite set of time dependent ODEs for the Fourier coefficients of $u(t)$. Since \eqref{EDP} reads componentwise as $\dot{\hat u}_n(t) = \hat u_{C(n)}(t)$, it naturally triggers, finding $\hat u_n(t)$ for each $n$, the corresponding Collatz iterations. The behavior of this set is nontrivial and may lead to interesting dynamics, as the Collatz operator does. Our first result will be the existence of this flow, and the characterization of some hypothesized long time Collatz orbits through particular solutions to \eqref{EDP}.

\begin{theorem}\label{MT}
The Collatz flow \eqref{EDP} is well defined and globally well-posed in $L^2(\mathbb T)$, and satisfies
\be\label{cotas}
\|u(t)\| \leq e^{\sqrt{2} t} \|u_0\|.
\ee
Moreover, we have the following alternatives:
\begin{itemize}
\item[(i)] Given a solution $u$ of \eqref{EDP}, there exists a nontrivial periodic orbit of the Collatz mapping $C$ if and only if there exists a coordinate $n_1$ of the solution map $(\hat u_n(t))_{n\in\mathbb Z}$ having the form 
\begin{equation}\label{sol_EDO0}
  \hat{u}_{n_1}(t)=c_1e^{\lambda_1t}+c_2e^{\lambda_2t}+\cdots+c_me^{\lambda_mt}, \quad c_j\in\mathbb C,
\end{equation}
where each $\lambda_j$ satisfies $|\lambda_j|=1$, $j=1,2,\ldots, m$, for some $m\geq 1$.
\item[(ii)]   There exist an unbounded sequence of Collatz $n_j=C(n_{j-1})$, $j\ge 1$, $n_j\to \infty$ if and only if the Collatz flow \eqref{EDP} has a solution $(\hat u _n(t))_{n\in\mathbb Z}$ with
\begin{equation*}
  \hat u_{n_j}(t)=\sum_{k\ge 0}\hat u_{n_{j+k}}(0) \frac{t^k}{k!}.
\end{equation*}
\end{itemize}
\end{theorem}

The previous result states that the evolution Collatz flow \eqref{EDP} is able to encode particular aspects of the Collatz mapping in a well-defined fashion. For instance, the existence of nontrivial periodic orbits is directly related to solutions of the ODE system having a substructure with essentially eigenvalues of absolute value 1. Therefore, the existence of these isolated solutions is an indication of periodic orbits in the original iterated Collatz mapping. Item $(ii)$ deals with the existence of an unbounded trajectory in the Collatz mapping, an open problem. In this case, we prove that unbounded Collatz orbits are directly related with exponentially growing solutions with a very particular structure, made of tails of exponential functions, and with an exponent always bounded by 1. In general, we believe that these two phenomena are directly related to a limiting spectral radius 1, which is below the expected spectral radius $\sqrt{2}$.

\subsection{Discrete Collatz time evolution}

In a second part of this paper, we want to study dynamical properties of the operator $C$. How dynamical properties are encoded or defined is an important issue. In this paper we have decided to face this problem in a different direction to previous works, this time seeking to mimic the PDE flow \eqref{EDP} but in the discrete setting. The consequences of these definitions, and the new results, are in our opinion of independent interest.

To state our second result, we need some definitions. First of all, recall the generalized Collatz mapping $C_{\alpha,\beta}$ introduced in \eqref{C_alpha}. For simplicity, we will avoid the parameters $\alpha$ and $\beta$ in the definition. Now, let us introduce the discrete time derivative of the Collatz operator at stage $k$:
\be\label{DiscreteD}
D_t(C^k(n)) := C^{k+1}(n)-C^k(n) = (C-I)(C^k(n)).
\ee
Finally, for $m\ge 1$, the $m$-th derivative of Collatz orbit at time $k$ is 
\[
D_t^m(C^k(n)) := D_t D_t^{m-1} (C^k(n)). % = \sum_{j = 0}^m \binom{m}{j}(-1)^jC^{k+m-j}(n).
\]
The idea behind the previous two definitions is to mimic the time derivative of the  previous continuous operator flow. Since the Collatz mapping is just discrete, \eqref{DiscreteD} assumes that one takes only one time step of size 1. Even this simple definition is enough to capture interesting consequences. In what follows, we shall study the discrete mapping
\[
C^k(n) \longmapsto D_t^m(C^k(n)), \quad k,m,n\geq 1.
\]
We will prove the following facts:

\begin{theorem}\label{MT2}
The Collatz orbit associated to the operator $C$ in \eqref{C_alpha} and the discrete time derivatives \eqref{DiscreteD} satisfy the following dynamical properties:
\begin{itemize}
\item[(i)]  Energy conservation. For any $n,k\in\mathbb N$, we have
\be\label{Energy_conservation}
	\sum_{j = 0}^{2^{k}-1} \left( C^k (n+j+2^k)-C^k \left(n+j \right) \right)= (1+\alpha)^k.
\ee
\item[(ii)] Bounds on the energy terms. If $k,m\in\mathbb N$, and 
\be\label{s_k} 
s_k (n)= \sum_{j = 0}^{2^k-1} C^k(n+j), \q s_{k,m} (n) = \sum_{j = 0}^{2^k-1}C^k(n+j+m2^k),
\ee
then if $\alpha\geq 3$,
\be\label{liminfsup}
1+\dfrac{2}{\alpha}m\leq  \liminf_{k \to \infty} \dfrac{s_{k,m}(n)}{s_k(n)} \leq \limsup_{k \to \infty} \dfrac{s_{k,m}(n)}{s_k(n)} \leq 1+ 2\alpha m.
 \ee
 In the case $\alpha=1$ we get
\be\label{liminfsup1}
1+\dfrac{2m}{1+2\beta}\leq  \liminf_{k \to \infty} \dfrac{s_{k,m}(n)}{s_k(n)} \leq \limsup_{k \to \infty} \dfrac{s_{k,m}(n)}{s_k(n)} \leq 1+ 2 m.
 \ee 
\item[(iii)] Let $m,n,k \in \NN$. Then $D_t^m(C^k(n))$ decomposes as a sum of terms with factor $C^k(n)$ and others with no factor at all. Moreover, the sum of the coefficients associated to $C^k(n)$ is $(\alpha -3)^m$, and $\beta (\alpha -3)^{m-1}$ for the free ones, respectively.
\end{itemize}
\end{theorem}

We interpret \eqref{Energy_conservation} as a sort of energy conservation, since the right hand side does not depend on the point $n$. The dependence on $k$ agrees with the fact that the Collatz flow has also exponential growth (Theorem \ref{MT}), assuming that $k$ is understood as a sort of discrete time variable. The case $\alpha=3$ (Collatz) in \eqref{Energy_conservation} is sharp, in the sense that after average by shift and number of terms of the sum, one gets exactly one (see \eqref{pseudo_virial}). If $\alpha\geq 5$, then such a result is lost and the fully averaged energy seems to grow with $k$.  Additionally, the result is independent of $\beta$, a fact that has some undesired consequences.  Item $(ii)$ measures the relative sizes of the components in the energy formula \eqref{Energy_conservation}. It is desirable to have a suitable control on each quantity relative each other, and \eqref{liminfsup} precisely gives a concise bound on the ratio of each member.

One of the most important consequences of Theorem \ref{MT2} is item $(iii)$, which roughly establishes that if every option in the Collatz flow has the same probability, then the $m$ discrete time derivative of the Collatz $k$ stage operator $C^k$ has zero sum of its coefficients in the Collatz case $\alpha=3$. In the Nollatz case $\alpha=1$, the sum converges to zero as $m$ tends to infinity, and it diverges if $\alpha>3.$

\subsection*{Organization of this paper} This paper is organized as follows. Section \ref{2} deals with functional properties of the Collatz operator, and the proof of local and global existence. Theorem \ref{MT} is proved in this part. Section \ref{3} considers the discrete case (Theorem \ref{MT2}), proving the energy formula \eqref{Energy_conservation} and the identity related to the time derivatives. 

\subsection*{Acknowledgments} All authors of this paper were supported by Fondecyt Exploraci\'on 13220060. ANID's support in this direction is warmly acknowledged.

\section{Functional properties. Existence of the Collatz flow}\label{2}

\subsection{Basic properties} In this section, we show that the flow defined in \eqref{EDP} is well-defined and global in time. Since the operator was defined for indices in $\mathbb Z$ in an odd fashion, it is enough to consider $n\geq 0$ in \eqref{C}. Even the case $n=0$, where $C(0)=0$, does not produce any change. Therefore, we will consider only the case $n\geq 1$ in forthcoming computations. First of all, recall that the Collatz operator \eqref{C} is linear and bounded. Indeed, the linearity of the operator is straightforward; let us now prove that it is a bounded operator. The boundedness follows from the classical fact
\[ 
\begin{array}{rl}
{\displaystyle \|\mathbf{C}(u)\|^2 =}&{\displaystyle \sum_{n\in \mathbb{N}} |\hat u_{C(n)}|^2 = \sum_{n=2k} |\hat u_{C(n)}|^2+\sum_{n=2k+1} |\hat u_{C(n)}|^2} \vspace{0.1cm}\\
  =& {\displaystyle \sum_{n\in\mathbb N} |\hat u_{n}|^2+\sum_{n=2k+1} |\hat u_{C(n)}|^2} \vspace{0.1cm}\\
  =& {\displaystyle \| u\|^2 +\sum_{n=2k+1}| \hat u_{C(n)}|^2  \le  2\| u\|^2}.
\end{array}
\]
Thus
\begin{equation}\label{normC}
\|u\|  \leq  \|\mathbf{C}(u)\|  \le \sqrt{2}\|u\| ,
\end{equation}
and hence $\mathbf{C}$ is bounded. Moreover, from the fact that $\mathbf{C}$ is a linear bounded operator, one has that its Gateaux derivative satisfies $\mathbf{C}'(u)(v) =\mathbf{C}(v)$. Additional functional properties can be found in \cite{Mori24}, where in particular eigenvalues and eigenfunctions are considered. Notice that $\mathbf{C}(u_0)=\lambda u_0$ leads to a solution to \eqref{EDP} of the form $u(t)= u_0 e^{\lambda t}$. It is easily proved (and probably well-known), that $|\lambda| =1$. Indeed, first of all, from $\|u\|  \leq  \|\mathbf{C}(u)\| $ one has $|\lambda|\geq 1$. Second, if $\mathbf{C}(u_0)=\lambda u_0$ and $u_0\in L^2$ is nontrivial, then $\lambda u_{0,n} = u_{0,C(n)}$ for all $n\geq 0$. Let us take a coordinate $u_{0,n_0} \neq 0$.  If the sequence $(C^k(n_0))_{k\in \mathbb N}$ diverges, then $u_{0,C^k(n_0)} = \lambda^k u_{0,n_0}$, and %since the sequence diverges,
\[
\| u_0\| ^2 \geq \sum_{k\geq 0} |u_{0,n_0}|^2 |\lambda|^{2k} =+\infty.
\]
This is clearly an absurd and this case does not exist.  Now, we assume the existence of bounded sequences. In this case, if the sequence $(C^k(n_0))_{k\in \mathbb N}$ converges to a (trivial or nontrivial) cycle of length $m$, it is known that inside that cycle the equation $\lambda u_{0,n} = u_{0,C(n)}$ leads to the condition $\lambda^m =1$, proving that $|\lambda|=1$.   

The existence of eigenfunctions and eigenvalues for $\mathbf{C}$ is an interesting problem. In the classical Collatz case, $\alpha=3, \beta=1$, the vector $u_0=(u_{0,1},u_{0,2},u_{0,3},u_{0,4},\ldots)=(1,1,0,0,\ldots)^T$ satisfies $\mathbf{C}(u_0)_n = u_{0,C(n)}= (1,1,0,0,\ldots)^T$. In general, cycles induce eigenfunctions with ones placed on the coordinates associated to the cycle.

\begin{remark} Let $ S:= \{ u \in L^2(\mathbb T) ~ :  ~ \hat u_{m} =0, \; 2 m \equiv \beta \bmod \alpha  \}$. Then 
$\|\mathbf{C} u\| = \|u\|$ for all $u\in S$. Indeed, every frequency $m$ has one even pre-image $2 m$, and it gets a second (odd) pre-image only when $2 m \equiv \beta\bmod \alpha$. Hence a coefficient $\hat{u}_m$ contributes $\left|\hat{u}_m\right|^2$ to $\left\|\mathbf{C} u\right\|^2$ in the first case and $2\left|\hat{u}_m\right|^2$ in the second. In the definition of $S$ we have removed exactly the indices that would double, so the remaining ones contribute once each: $\left\|\mathbf{C} u\right\|^2=$ $\sum_m\left|\hat{u}_m\right|^2=\|u\|^2$. Thus $\left\|\mathbf{C} u\right\|=\|u\|$ for every $u \in S$. Now if   
  \(\displaystyle T:=\{u\in L^{2}(\mathbb T):\hat u_{m}=0\ \text{whenever }2m\not\equiv\beta\bmod{\alpha}\}\),  
  each non-zero coefficient is counted twice, giving  
  \(\|\mathbf{C}u\|^2 = 2\|u\|^2 \), hence \( \|\mathbf{C}u\|=\sqrt{2}\,\|u\|\) for all \(u\in T\). This also proves that the operator norm of $ \mathbf{C}$ satisfies $\| \mathbf{C} \| =\sqrt{2}.$
\end{remark}

\subsection{Collatz operator on $\ell^2$}
Let us generalize the definition \eqref{C} and prove some interesting functional analysis results. We introduce the Collatz operator $\textbf{C}: \ell^2 \to \ell^2$ by 
\be\label{C2}
\textbf{C}\left(\sum_{n = 1}^\infty u_ne_n\right) = \sum_{n = 1}^\infty u_{C(n)}e_n,
\ee
where $\{e_n\}_n$ is the canonical basis of $\ell^2$. In our previous definition, $e_n (x)=e^{in\pi x}$. We recall that this operator satisfies (compare with \eqref{normC})
\be\label{C3}
\|u\| \leq \|\textbf{C}(u)\| \leq \sqrt{2}\|u\|.
\ee
Indeed
\begin{align*}
 \|\mathbf{C} u\|^2=\sum_{k \in \mathbb{Z}} |u_{C(k)}|^2 \geq \sum_{2j \in \mathbb{Z}} |u_{C(2j)}|^2 = \sum_{j \in \mathbb{Z}} |u_j|^2=\|u\|^2 .
\end{align*}
%In this section will show some basic properties of this operator
It is clear that 
%\begin{lemma}[Adjoint operator]
the adjoint operator of $\textbf{C}$ is given by 
\[
\textbf{C}^*\left( \sum_{n = 1}^\infty u_ne_n \right) = \sum_{n = 1}^\infty u_ne_{C(n)}.
\]
Indeed, let $u, v \in  \ell^2$. By the symmetry given in \eqref{simetriaC}  we focus in $n\in \mathbb N$. First note that for each $n\in \mathbb N$ 
\[
u_{C(n)}v_n=u_n\left(\sum_{j \in C^{-1}(n)}v_j\right).
\]
In particular have that 
\begin{eqnarray*}
\langle \textbf{C}u, v \rangle = \sum_{n = 1}^\infty u_{C(n)}v_n 
= \sum_{n = 1}^\infty u_n\left(\sum_{j \in C^{-1}(n)}v_j\right) 
\end{eqnarray*}
Noticing that
\[
\sum_{n = 1}^\infty \left(\sum_{j \in C^{-1}(n)}v_j\right)e_n = \sum_{n = 1}^\infty v_ne_{C(n)},
\]
we conclude.

\begin{proposition}
For the Collatz operator defined in \eqref{C2} the following properties hold:
\begin{enumerate}
\item $\mathbf{C}$  is injective and has closed range.
\item The operator $\mathbf{C}^*\mathbf{C}$ is bijective.
\item The kernel of $\mathbf{C}^*$ has infinite dimension. Moreover, in the classical Collatz case, the set $\{e_n-e_{\alpha n+\beta}\}_{n \in 2\mathbb{N}-1}$ is an orthogonal basis of the kernel.
\end{enumerate}
\end{proposition}

\begin{proof}
Assertion (1) follows from the inequality $\|u\| \leq \|\mathbf{C}(u)\|$ in \eqref{C3}. Note that the bilinear form $a(u,v) = \langle \mathbf{C}^*\mathbf{C}u,v \rangle$ is continuous and coercive since 
\[
\langle \mathbf{C}^*\mathbf{C}u,u \rangle = \langle \mathbf{C}u,\mathbf{C}u \rangle = \|\mathbf{C}u\|^2 \geq \|u\|^2,
\]
and then from Lax-Milgram Theorem, assertion (1) and (2) hold. {\color{black} For (3) we can see easily that 
\[
\{e_n-e_{\alpha n+\beta}\}_{n \in 2\mathbb{N}-1} \subseteq \ker \mathbf{C}^* .
\]
If $n \neq m$ are odd natural numbers then
\[%begin{eqnarray*}
\langle e_n-e_{\alpha n+\beta}, e_m-e_{\alpha m+\beta}\rangle = \langle e_n, e_{\alpha m+\beta} \rangle +\langle e_{\alpha n+\beta}, e_m\rangle ,
\]%end{eqnarray*}
and these products are non-zero only when $n = \alpha m+\beta$ and $m = \alpha n+\beta$, but since $n,m$ were supposed to be odd then $\alpha n +\beta$ and $\alpha m +\beta$ are even, so the set is orthogonal. In order to see that this set generates $\ker \mathbf{C}^*$ let $u = \sum_{n = 1}^\infty u_ne_n$, and use that 
\begin{eqnarray*}
\mathbf{C}^*u = \sum_{n_ = 1}^\infty u_ne_{C(n)} = \sum_{n = 1}^\infty \left(\sum_{j \in C^{-1}(n)}u_j\right)e_n.
\end{eqnarray*}
Then $\mathbf{C}^*u = 0$ implies that 
\begin{itemize}
\item if $|C^{-1}(n)| = 1$ then $u_{C^{-1}(n)} = 0$,
\item if $|C^{-1}(n)| = 2$ then for $j_1, j_2 \in C^{-1}(n)$, $u_{j_1} +u_{j_2} = 0$.
\end{itemize}
In the second case, we have that \[
n = \frac{j_1}{2} = \frac{\alpha j_2+\beta}{2},\]
 and $u$ is of the form 
\[
u = \sum_{j \in 2\mathbb{N}-1}u_j(e_j-e_{\alpha j+\beta}),
\] 
which completes the proof.}
\end{proof}
\begin{remark}
A classical result of functional analysis shows that in the previous result, (1) implies that $\mathbf{C}^*$ is surjective, and (2) implies that neither $\mathbf{C}$ nor $\mathbf{C}^*$ can be compact.
\end{remark}

\subsection{Local and global existence of the Collatz PDE flow} 
Now we prove the first part of Theorem \ref{MT}, including the bound \eqref{cotas}. 

\begin{theorem}[Existence and uniqueness]
Let $C$ be the Collatz operator defined in \eqref{C}. There exists a unique solution 
$u$ to the equation \eqref{EDP} defined globally in time. Moreover,
\be\label{CotaExpo}
\|u(t)\|  \leq e^{\sqrt{2}t}\|u(0)\| .
\ee 
\end{theorem}

\begin{proof}
We will show the existence via a fixed point approach. Note that, by integrating both sides of the equation \eqref{EDP}, we obtain
\be\label{Duhamel}
u(t)=u_0+\int_0^t C(u)(s)ds. 
\ee
Recall that $\mathbb T= \left[ -\pi, \pi \right]$ with periodic boundary conditions. Let $t^*$ be a fixed final time, and $X=C\left( \left[ 0, t^* \right], L^2 (\mathbb T) \right)$ and  $F:X\to X$  the operator  described by
\[
\begin{array}{rl}
%F: & X \longrightarrow X \\
   & \displaystyle u  \longmapsto u_0 + \int_0^t \mathbf{C}(u)(s,x)ds.
\end{array}
\]
 Note that if $u\in X$, 
\[
t \mapsto u(t, \cdot ) \in L^2 \quad  \hbox{and} \quad  \sup_{s\in \left[0, t^{\ast }\right] } \|u(s)\|  <\infty.
\]
We can see that $F$ is well defined and is a contraction. Indeed, let $t$ be fixed, $F(u)(s)\in L^2$. Then having in mind the integral Minkowski inequality one has
\[
\ba
\|Fu(s)\|  & \leq   \|u_0\| +\left( \int_{-\pi}^\pi\left| \int_0^t \mathbf{C}(u)(s,x)ds\right|^2 dx\right)^{\frac{1}{2}}  \\
& \leq   \|u_0\| +\int_0^t\left( \int_{-\pi}^\pi\left| \mathbf{C}(u)(s,x)\right|^2dx\right)^{\frac{1}{2}} ds \\
& \leq   \|u_0\| +\int_0^t \sqrt{2}\| u(s)\|  ds \\
& \leq  \|u_0\| +\sqrt{2}t^{\ast } \sup_{t\in \left[0, t^{\ast }\right] } \|u(s)\|   < \infty.
\ea
\]
In order to prove the continuity of $F$,  let us consider $t_1, t_2 \in \left[ 0, t^{\ast }\right]$, with $t_2 < t_1$. Thus,  one has
\[ 
\ba
\| Fu(t_1)-Fu(t_2)\|  & = \left( \int_0^t \left| \int_0^{t_1} \mathbf{C}(u)(s)ds -\int_0^{t_2} \mathbf{C}(u)(s)ds \right|^2dx \right)^{\frac{1}{2}} \\
& \le  \left( \int_{-\pi}^\pi \left| \int_{t_2}^{t_1} \mathbf{C}(u)(s)ds \right|^2dx\right)^{\frac{1}{2}}  \\
& \leq \int_{t_2}^{t_1}\left(   \int_{-\pi}^\pi \left| \mathbf{C}(u)(s) \right|^2 dx\right)^{\frac{1}{2}}ds \\
& =\int_{t_2}^{t_1}\| \mathbf{C}(u) \| ds  \leq \int_{t_2}^{t_1} \sqrt{2} \| u(s) \|  ds  \leq  2(t_1-t_2)M. 
\ea
\]
where ${\displaystyle M=\sup_{s\in \left[0, t^{\ast }\right] }\|u(s) \| <\infty.}$

\medskip

On the other hand, we want to prove that $F$ is a contraction. Let $u_1,u_2$ be in $X$. We need to prove that $d_X(Fu_1,Fu_2)\leq K d_X (u_1,u_2)$, for some $0<K<1$. Indeed, for $t\geq 0$ and by using Minkowski's inequality, we obtain
\[
\ba
& \left\| \int_0^t \mathbf{C}(u_1)(s)ds-\int_0^t \mathbf{C}(u_2)(s)ds \right\| \\
& =   \left\| \int_0^t \mathbf{C}(u_1)(s)-\mathbf{C}(u_2)(s)ds\right\|   = \left( \int_{-\pi}^\pi \left|\int_0^t \mathbf{C}(u_1)(s)-\mathbf{C}(u_2)(s)ds\right|^2 dx\right)^{\frac{1}{2}} \\
& \leq  \int_0^t \left(  \int_{-\pi}^\pi \left|\mathbf{C}(u_1)(s)-\mathbf{C}(u_2)(s) \right|^2 dx \right)^{\frac{1}{2}}ds\\
& \leq  \int_0^t \sqrt{2} \|u_1(s)-u_2(s)\|  ds \leq   \sqrt{2}  t^{\ast }\sup_{s\in \left[ 0, t^{\ast }\right]} \|u_1(s)-u_2(s)\|    =  \sqrt{2}  t^{\ast } d_X (u_1, u_2).
\ea
\]
taking $t^{\ast }<\frac12$, we have that $F$ is a contraction. We conclude that $F$ has a fixed point which for us represent the solution $u$ of \eqref{EDP}, associated to $u_0$. Finally, the global existence is a direct consequence of the fact that $C$ is globally Lipschitz
\[
\|\mathbf{C}(u_1)-\mathbf{C}(u_2)\|  \leq \sqrt{2} \|u_1-u_2\| .
\]
Now we prove \eqref{CotaExpo}. From \eqref{Duhamel} and \eqref{C3},
\[
\|u(t)\| \leq \|u_0\| + \int_0^t \|\mathbf{C}(u)(s)\|ds \leq \|u_0\| + \int_0^t \sqrt{2} \|u(s)\|ds. % \leq  
\]
Gronwall's inequality leads to \eqref{CotaExpo} (it also corroborates that Collatz' solutions are globally defined).
\end{proof}

\subsection{Dynamics of key solutions}

In this section we describe how particular cases of the Collatz dynamics translate into equivalent dynamics in the Collatz PDE flow. The following result proves item $(i)$ in Theorem \ref{MT}.

\begin{theorem}\label{Thm:per} Given a solution $u$ of \eqref{EDP} associated with an initial condition, there exists a nontrivial periodic orbit of the Collatz map $C$ if and only if there exists a coordinate of the form 
\begin{equation}\label{sol_EDO}
  \hat{u}_{n_1}=c_1e^{\lambda_1t}+c_2e^{\lambda_2t}+\cdots+c_me^{\lambda_mt}
\end{equation}
where $|\lambda_j|=1$.
\end{theorem}

Notice that \eqref{sol_EDO} exactly corresponds to \eqref{sol_EDO0}.
\begin{proof}
The implication from left to right follows from the following argument. Let us consider a nontrivial Collatz cycle 
\[
n_1\to n_2=C(n_1)\to \cdots \to n_m=C^m(n_1)=n_1.
\]
Consider the initial Collatz system \eqref{EDP}. Then, for each $i\in \left\lbrace 1, ..., m\right\rbrace$ one has  
\begin{equation}\label{subsisind}
\begin{cases}
\partial_t \hat{u}_{n_i} = \hat u_{C{(n_i)}}, \, \text{if } 1\leq i \leq m-1, \\
\partial_t \hat{u}_{n_i} =\hat u_{n_1}, \, \text{if }  i=m.
\end{cases}
\end{equation}
Thus $\partial_t \hat{u}_{n_1} = \hat u_{2}$, $\partial_t^2 \hat{u}_{n_1} = \hat u_{3}$, \ldots $\partial_t^m \hat{u}_{n_1}  =\hat{u}_{n_1}$.  Now $\partial_t^m \hat{u}_{n_1}  =\hat{u}_{n_1}$ implies the eigenvalue equation satisfies $\lambda^m=1$. Therefore, every root of this equation must satisfy $|\lambda|=1$. Let $\lambda_j \in \mathbb C$, $j=1,\ldots,m$, be these roots, counting multiplicity.  Thus  $\hat{u}_{n_1}$ is combination of complex exponential of the form 
\[%begin{equation}
  \hat{u}_{n_1}=c_1e^{\lambda_1t}+c_2e^{\lambda_2t}+\cdots+c_me^{\lambda_mt}.
\]%end{equation}
Conversely, let $u$ the solution of \eqref{EDP} such that \eqref{subsisind} is satisfied, and 
where $C^k(n_1)\neq n_1 $ if $k=1,\cdots, m-1$. Let $n_j=C(n_{j-1})$ where $j=1,\cdots m-1$. Clearly $n_j=C^{j-1}(n_1)\neq n_1$ for $j=1,\cdots m-1$. Suppose that $C^{i\cdot m}(n_1)\neq n_1$ for all $i\in \mathbb{N}$. From \eqref{subsisind} we have for each $i$, $\partial_t^{i m} \hat{u}_{n_1} =\hat u_{C^{i m}(n_1)} =\hat{u}_{n_1}$. Note that taking in mind the positions $C^{im}(n_1)$ of $\hat u$  we obtain: 
\begin{align*}
  \|u\| ^2&\ge \sum_{i=1}^{\infty} |\hat{u}_{C^{im}(n_1)}|^2= \sum_{i=1}^{\infty} |\hat{u}_{n_1}|^2=\infty,
\end{align*}
a contradiction since $u\in L^2$. Thus there is $i_0\in \mathbb{N}$ such that $C^{i_0 m}(n_1)=n_1.$

Let $u$ a solution of \eqref{EDP} such that
\begin{equation}\label{n1}
  \hat{u}_{n_1}=c_1e^{\lambda_1t}+c_2e^{\lambda_2t}+\cdots+c_me^{\lambda_mt}.
\end{equation}
Note that $\partial_t\hat{u}_{n_1}=\hat{u}_{C(n_1)}$. Denotes $n_2=C(n_1)$ then we have $\partial_t\hat{u}_{n_2}=\hat{u}_{C(n_2)}$ in the last step by \eqref{n1} we obtain $\partial_t\hat{u}_{n_m}=\hat{u}_{C(n_m)}=\hat{u}_{C^m(n_1)}=\hat{u}_{n_1}$.
\end{proof}

\begin{lemma}\label{previa} Let $u$ be a solution of \eqref{EDP}. Suppose there is a divergent orbit 
\be\label{secue}
n_1\to n_2:=C(n_1)\to n_3:=C^2(n_1)\cdots\to n_j:=C^{j-1}(n_1)\to \cdots
\ee
Then 
\begin{equation}\label{diver}
\hat u_{n_j}(t)= \left\langle  \left( \hat u_{n_k} (0) \right)_{k\geq j} , \left( \frac{t^k}{k!}\right)_{k\in\mathbb N}\right\rangle_{\ell^2(\mathbb N)}.
\end{equation} 
\end{lemma}

\begin{proof}
Let $t\in\mathbb R$.  For simplicity we denote  
 \begin{equation}\label{OrbExp}
 \text{Orb}_j(u)= \left( \hat u_{n_k} (0) \right)_{k\geq j}, \quad %\left(\hat u_{n_j}(0), \hat u_{u_{n_{j+1}}}(0), \hat u_{n_{j+2}}(0), \hat u_{n_{j+3}}(0), \cdots\right).
 \text{Exp}_\infty(t)= \left( \frac{t^k}{k!}  \right)_{k\in \mathbb N}.
\end{equation}  
Both are elements in $\ell^2 (\mathbb N)$. Let $\hat v(t):=\left(\hat v_{j}(t) \right)_{j\in\mathbb N}$, with
\[
\begin{cases}
    \hat v_{i}(t)=  \left\langle   \text{Orb}_j(u) ,  \text{Exp}_\infty(t) \right\rangle_{\ell^2(\mathbb N)},  & \hbox{$i=n_j$;} \\ %\left(\hat u_{n_j}(0), \hat u_{u_{n_{j+1}}}(0), \hat u_{n_{j+2}}(0), \hat u_{n_{j+3}}(0), \cdots\right) \cdot\left(1, t, \frac{t^2}{2}, \frac{t^3}{3!}, \frac{t^4}{4!}, \cdots\right), & \hbox{$i=n_j$;} \\
    \hat v_i(t)=\hat u_i(t), & \hbox{$i\neq n_j$.}
\end{cases}
\]
%%\begin{equation} v_{n_j}(t)=\left(\hat u_{n_j}(0), \hat u_{u_{n_{j+1}}}(0), \hat u_{n_{j+2}}(0), \hat u_{n_{j+3}}(0), \cdots\right) \cdot\left(1, t, \frac{t^2}{2}, \frac{t^3}{3!}, \frac{t^4}{4!}, \cdots\right)\end{equation} 
Note that for each $t\ge 0$ the number $\hat v_{n_j}(t)\in \mathbb{C}$  is well defined since 
\[
\ba
|\hat v_{n_j}(t)|\le &~{} \|\text{Orb}_j(u)\|_{\ell^2} \| \text{Exp}_\infty(t)\|_{\ell^2} \\
\leq &~{} C \| u(0)\|  \| \text{Exp}_\infty(t)\|_{\ell^1} \leq C  \| u(0)\|  e^{t}.
\ea
\]  
Since  $u$ is a solution of \eqref{EDP},  and having in mind \eqref{secue}, we obtain the following system %of infinity equations 
\begin{equation}\label{system1}
\partial_t \hat{u}_{n_j}  =\hat{u}_{C(n_j)} =\hat{u}_{n_{j+1}}, \quad j\in \mathbb{N}.
\end{equation}
On the other hand, using \eqref{diver} and \eqref{OrbExp}, 
\[
\ba
 \partial_t \hat{v}_{n_j}= &~{} \left\langle   \text{Orb}_j(u) ,  \partial_t \text{Exp}_\infty(t) \right\rangle_{\ell^2(\mathbb N)} =\left\langle   \text{Orb}_{j+1}(u) ,   \text{Exp}_\infty(t)    \right\rangle_{\ell^2(\mathbb N)} = \hat v_{n_{j+1}}. %Orb^*_j(u)\partial_t \text{Exp}^*_\infty(t)=Orb^*_j(u)\cdot \left(0,1, t, \frac{t^2}{2}, \frac{t^3}{3!}, \frac{t^4}{4!}, \cdots\right)=\hat v_{n_{j+1}}
\ea
\]
Thus $\hat v$ is a solution of system \eqref{system1} and by uniqueness of the system \eqref{system1} for each $j\in \mathbb N$    $v_{n_j}=u_{n_j}$. Hence \eqref{diver} is obtained.                               
\end{proof}

As a corollary of the previous result we have item $(ii)$ in Theorem \ref{MT}.

\begin{corollary}
  There exists an unbounded sequence of Collatz $n_j=C(n_{j-1})$, $j\ge 1$, $n_j\to \infty$ if and only if the Collatz flow has a solution $(\hat u _n(t))_{n\ge 0}$ where in particular
\begin{equation}\label{forma}
  \hat u_{n_j}(t)=\sum_{k\ge 0}\hat u_{n_{j+k}}(0) \frac{t^k}{k!}.
\end{equation}
\end{corollary}

The previous formula establishes that, in the case of a diverging Collatz orbit, the corresponding coordinates in the Collatz flow will have better growth estimates, since the starting time independent coefficient in \eqref{forma} is the one at time zero that is at the position $n_{j+0}$. The remaining coefficients will be of type $n_{j+k}$, namely, growing in size, and since the original data is in $L^2$, they will naturally decay faster than expected. We can conclude then that the solution, at the coordinates indicated by the hypothesized Collatz unbounded sequence, will behave better than normal solutions which have strong mixing and decay to the trivial cycle. 

\begin{proof}
The necessary condition is consequence of Lemma \ref{previa}, the nonperiodic character of the sequence $C^j(n_1)$, and the fact that convergence is uniform for each $t$,
\[
\sum_{j\ge 1}  |\hat u_{n_j}(t)|^2 =\sum_{k,k' \ge 0} \frac{t^{k+k'}}{k!k'!}   \sum_{j\ge 1}  \hat u_{n_{j+k}}(0) \overline{\hat u_{n_{j+k'}}(0) } \leq e^{2t} \| u(0)\| ^2.
\]
The remaining terms of the series are naturally bounded by $e^{\sqrt{2} t}$ as in \eqref{CotaExpo}. As for the sufficient condition, we use \eqref{system1} as follows: by hypothesis \eqref{forma}, $\partial_t \hat{u}_{n_j}  =\hat{u}_{n_{j+1}}$, and then \eqref{secue} leads to $\partial_t \hat{u}_{n_j} =\hat{u}_{C(n_j)} $. For the rest of coordinates, we simply solve $\partial_t \hat{u}_{k} =\hat{u}_{C(k)}$ with a particular initial condition related to the coordinate $k$. The proof is complete. 
\end{proof}
Now we finally consider the most common state of art according to existing numerics, which is the case of a number $n$ converging through Collatz iterations, after $\ell$ steps, into a cycle of size $k$. This is the case of the classical Collatz operator with $\alpha=3$ and $\beta=1$, where the trivial cycle is $1\mapsto 2$ ($k=2$), and for any natural number up to $10^{20}$.  

\begin{lemma}\label{clasico} Let $n\in\mathbb N$ and consider the associated Collatz sequence $(n_j)_{j\geq 1}$, $n_{j+1}=C(n_j)$, $j\geq 1$ achieving for the first time the cycle of size $m$ after $\ell$ steps. Then the associated coordinate $\hat u_n(t)$ of the corresponding PDE flow \eqref{EDP} satisfies 
\begin{equation}\label{desenrrollar}
\ba
\hat u_n(t)= &~{} \sum_{k=0}^{\ell-1}  \left( \hat{u}_{\hat n_{k}} (0) -  \sum_{j=1}^m \frac{c_j}{\lambda_j^{\ell-k}} \right)\frac{t^{k}}{k!}  + \sum_{j=1}^m \frac{c_j}{\lambda_j^\ell}e^{\lambda_jt}.
\ea
\end{equation}
with $\hat n_0=n$, $ c_j \in \mathbb{C}$, $\left|\lambda_j\right|=1$, $j=1,2, \ldots, m$ and some $m \geq 1$.
\end{lemma}

\begin{proof}
Following the first part of the proof of Theorem \ref{Thm:per}, let us consider a Collatz cycle 
\[
n_1\to n_2=C(n_1)\to \cdots \to n_m=C^m(n_1)=n_1.
\]
Thus  $\hat{u}_{n_1}$ is combination of complex exponential of the form 
\[%begin{equation}
  \hat{u}_{n_1}(t)=c_1e^{\lambda_1t}+c_2e^{\lambda_2t}+\cdots+c_me^{\lambda_mt}.
\]%end{equation}
Now we notice that $n_1=C^\ell(n)$. Let $\hat n_0:= n\to \hat n_1:= C(n)\to \hat n_2:= C(\hat n_1) \to \cdots \to \hat n_{\ell-1} := C(\hat n_{\ell-2}) \to  \hat n_\ell := C(\hat n_{\ell-1}) = n_1$ be the sequence leading to $n_1$ from $n$ in $\ell$ steps. Then $\partial_t \hat{u}_{\hat n_{\ell-1}} = \hat u_{n_1}$, $\partial_t \hat{u}_{\hat n_{\ell-2}} = \hat u_{\hat n_{\ell-1}}$ and so on up to $\partial_t \hat{u}_n = \hat u_{\hat n_1}$. Integrating in time once, and noticing that $|\lambda_j|=1$ for all $j$, we get
\[
\hat{u}_{\hat n_{\ell-1}} (t) =\hat{u}_{\hat n_{\ell-1}} (0) +  \frac{c_1}{\lambda_1}(e^{\lambda_1t}-1) +\frac{c_2}{\lambda_2}(e^{\lambda_2t}-1) +\cdots+  \frac{c_m}{\lambda_m}(e^{\lambda_mt}-1).
\]
Repeating this algorithm, and assuming $\ell \geq 2$,
\[
\hat{u}_{\hat n_{\ell-2}} (t) = \hat{u}_{\hat n_{\ell-2}} (0) +  \left( \hat{u}_{\hat n_{\ell-1}} (0)   -\sum_{j=1}^m \frac{c_j}{\lambda_j} \right)t  + \sum_{j=1}^m \frac{c_j}{\lambda_j^2}(e^{\lambda_jt}-1). % +\frac{c_2}{\lambda_2^2}(e^{\lambda_2t}-1) +\cdots+  \frac{c_m^2}{\lambda_m}(e^{\lambda_mt}-1).
\]
A third iteration, if allowed, gives
\[
\ba
\hat{u}_{\hat n_{\ell-3}} (t) = &~{} \hat{u}_{\hat n_{\ell-3}} (0) +  \left( \hat{u}_{\hat n_{\ell-2}} (0)   -\sum_{j=1}^m \frac{c_j}{\lambda_j^2} \right) t \\
&~{}  +  \frac12\left( \hat{u}_{\hat n_{\ell-1}} (0)   -\sum_{j=1}^m \frac{c_j}{\lambda_j} \right)t^2  + \sum_{j=1}^m \frac{c_j}{\lambda_j^3}(e^{\lambda_jt}-1). % +\frac{c_2}{\lambda_2^2}(e^{\lambda_2t}-1) +\cdots+  \frac{c_m^2}{\lambda_m}(e^{\lambda_mt}-1).
\ea
\]
Continuing $\ell-3$ times, we get \eqref{desenrrollar}.
\end{proof}

The previous result shows that pure polynomial behavior in Collatz evolutions is perfectly possible if the exponential part is made zero at the beginning in \eqref{desenrrollar}. In that sense, the conjectured bound
\[
\| u(t)\|   \geq c_0 e^{t} \| u(0)\| ,
\]  
coming from the fact that $\|{\bf C} (u)\| \geq \|u\|$, seems not valid. {\color{black} However, in some sense one can approach the maximum growth rate. Indeed, assume the classical Collatz operator with $\alpha=3$ and $\beta=1$. Let us choose the initial condition $u_0=\delta_{5,k}$, where $\delta_{j,k}$ is the Kronecker delta. We must then solve $\dot{\hat u}_k = \hat u_{C(k)}$. Then $u_1(t)=u_2(t)=0$ and for $k\geq 3$ one has $u_k(t) = \frac{t^{m_k}}{m_k!}$, where $m_k$ is the number of times under which $k$ achieves the value 5 under the Collatz application, i.e. $C^{m_k}(k)=5$. For instance, $m_{13}=3$, since $C^3(13)=5$. A crude estimate on the norm of $u(t)$ is given by
\[
\| u(t)\| ^2 =\sum_{k\geq 1} c_k \frac{t^{2m_k}}{(m_k!)^2} \leq \sum_{k\geq 0} 2^k \frac{t^{2k}}{(k!)^2},
\]
where $c_k$ is an integer between $0$ and $2^k$. The right hand side corresponds to $I_0(2\sqrt{2}t)$, where $I_0$ is the modified Bessel function of first kind. Its long time behavior is given by $I_0(x)\sim \frac{e^x}{\sqrt{2\pi x}}$ as $x\to +\infty.$ Consequently, $I_0(2\sqrt{2}t) e^{2\sqrt{2}t} \sim \frac1{\sqrt{t}}$, leading to an order that it is close to the maximal rate of growth.

}
\section{Discrete time evolution of the Collatz operator}\label{3}

In this section our objective is to prove Theorem \ref{MT2}. For this, we first start with a particular but classical formula. 

\subsection{Preliminaries} In the following we consider the accelerated Collatz type operator defined by \eqref{C_alpha}. 

\begin{lemma}\label{deco_coef}
For any $k \in \mathbb{N}$ there exist numbers $a(i,k), b(i,k) \in \mathbb{N}$, $i = 0, \cdots, 2^{k}-1$ such that 
\begin{equation}\label{modj}
C^k(n) = \dfrac{1}{2^k} \left( a(i,k)n+b(i,k) \right), \quad n \equiv i \mod 2^{k}.
\end{equation}
Moreover, the coefficients $a(i,k)$ are odd (in fact, they are powers of $\alpha$), and both $a(i,k)$ and $b(i,k)$ satisfy the following relations
\begin{equation}\label{sum}
\ba
& \sum_{i = 0}^{2^k-1}a(i,k) = (\alpha+1)^{k}, \\
&  \sum_{i = 0}^{2^k-1}b(i,k) = \begin{cases}
		\beta\left(\dfrac{(\alpha+1)^k-4^k}{\alpha +1-4}\right), & \alpha \neq 3 \vspace{2mm}\\
		\beta k (\alpha+1)^{k-1}  , & \alpha = 3. \end{cases} %\begin{cases}  \dfrac{(\alpha+1)^k-4^k}{\alpha-3} \\ \end{cases} .
\ea
\end{equation}
\end{lemma}

\begin{remark}
It will be clear from the proof of Lemma \ref{deco_coef} that that all the values of $ a(j,k) $ are powers of $\alpha$, but also independent of $\beta$. This fact will be important in forthcoming results below.
\end{remark}
\begin{proof}
	By induction on $k$, the case $k = 1$ clearly holds, because $C^1_{\alpha, \beta}(n) =C_{\alpha, \beta}(n)$ is given by \eqref{C_alpha}, and then we have $a(0,1)=1$ and $a(1,1)=\alpha$. Similarly $b(0,1)=0$ and $b(1,1)=\beta$.
	
	\medskip
	
	Now suppose  that the result holds for a fixed  $k \in \NN$. Note that for each $n \in \NN$ we can characterize its congruence class modulo $2^{k+1}$ knowing its congruence class modulo $2^k$ and the parity of $C^k(n)$. To see this take $p \in \NN$ such that 
	\[
	p2^{k+1} \leq n < (p+1)2^{k+1}.
	\] 
	If $n \equiv j \pmod{2^{k}}$ then $n$ can only take two possible values, namely for some $j \in \mathbb N\cap [0, 2^k-1]$
	\[ 
	n = p2^{k+1}+j \quad \text{or} \quad n = p2^{k+1}+j+2^k,
	\]
	but we know that, since the two last values for $n$ differ by $2^k$, when evaluating $C^k(n)$ we will obtain different parities (see \eqref{parity_seq}). Thus, using the induction hypothesis we can evaluate $C^{k+1}(n)$ for $n \equiv j \pmod{2^{k}}$, obtaining
	\[ 
	C^{k+1}(n) = \begin{cases}
		\dfrac{a(j,k) n+b(j,k)}{2^{k+1}}, &  C^{k}(n) \equiv 0 \hspace{-2mm}\pmod{2} \vspace{0.2cm} \\
		\dfrac{\alpha  a(j,k)  n+\alpha  b(j,k) +2^{k}\beta}{2^{k+1}}, &  C^{k}(n) \equiv 1 \pmod{2}.
	\end{cases}
	\]
	In this way, we obtain the coefficients for $C^{k+1}$, which are given by 
	\begin{align}\label{coef a}
		a(j,k+1) = \begin{cases}
			a(j,k), & C^k(n) \equiv 0 \pmod{2} \\
			\alpha a(j,k), & C^k(n) \equiv 1 \pmod{2}.
		\end{cases}
	\end{align}
	Also,
	\begin{align}\label{coef b}
		b(j,k+1) = \begin{cases}
			b(j,k), & C^k(n) \equiv 0 \pmod{2} \\
			\alpha  b(j,k) +2^{k}\beta, & C^k(n) \equiv 1 \pmod{2}.
		\end{cases}
	\end{align}
	Since $a(0,1) = 1, a(1,1) = \alpha$, it follows that $a(j,k)$ is always a power of $\alpha$. The precise value of this power is not needed, only a particular sum that will be computed now.
	This proves \eqref{modj} and the first part of the lemma. 
	
	Regarding \eqref{sum}, we proceed again by induction. Assume 
	\[ 
	\sum_{j = 0}^{2^k-1}a(j,k) = (\alpha+1)^k,
	\]then summing all the coefficients in \eqref{coef a} we obtain
	\[
	\sum_{j = 0}^{2^{k+1}-1}a(j,k+1) = (\alpha+1)\sum_{j = 0}^{2^{k}-1}a_{j}^{k} = (\alpha +1)^{k+1}.
	\]%end{eqnarray*}
	This proves the first identity in \eqref{sum}. Finally, if we denote 
	\[ 
	s_k = \sum_{j = 0}^{2^k-1}b(j,k),
	\]
	taking the sum of the coefficients in \eqref{coef b} we see that this sequence must satisfy the following recurrence
	\[ 
	s_{k+1}= \sum_{j = 0}^{2^{k+1}-1}b(j,k+1) = (\alpha+1)s_k+4^k\beta.
	\]Since we have $s_1 = \beta$ the solution of the recurrence is given by
	\[
	s_k = \begin{cases}
		\beta\left(\dfrac{(\alpha+1)^k-4^k}{\alpha +1-4}\right), & \alpha \neq 3 \vspace{2mm}\\
		\beta k (\alpha+1)^{k-1}  , & \alpha = 3.
	\end{cases}
	\]
	This finishes the proof of \eqref{sum} and the lemma.
\end{proof}

Let \(\mathcal P:\mathbb N\to\{0,1\}\) be the parity map, \(\mathcal P(x):=x\bmod 2\).
Following Terras \cite{Terras}, the (length-\(k\)) \emph{parity vector} of \(n\) with respect to the  map \(C\) is  
\begin{equation}\label{parity_seq}
\mathcal P_k(n):=\bigl(\mathcal P(n),
                        \mathcal P\bigl(C(n)\bigr),
                        \dots,
                        \mathcal P\bigl(C^{\,k-1}(n)\bigr)\bigr)
\in\{0,1\}^{k}.
\end{equation}

\begin{lemma}
The mapping $\mathcal{P}_k:[1,2^{k}] \cap \mathbb N \longrightarrow \{0,1\}^{k}$, defined in \eqref{parity_seq} is bijective.
\end{lemma}

\begin{proof}
For \(\alpha=3,\;\beta=1\) the claim is exactly
\cite[Cor.~1.3]{Terras}. For the remaining  cases Terras’ argument still yields. Indeed for every  $n\in \mathbb{N}$ we have
\be\label{ter}  \mathcal P_{k}(n+2^k)=\mathcal P_{k}(n)\quad \mathcal P(C^k(n+2^k))=1-\mathcal P(C^k(n+2^k)).\ee Starting from \(n_0=1\) (whose vector is \(\mathcal P_k(n_0)\)),
property \eqref{ter} shows that adding \(2^{a}\) flips exactly the
\(a\)-th component of \(\mathcal P_k\) and leaves the earlier ones unchanged.
Hence, by adding (or not) each power \(2^{a}\) with \(0\le a<k\),
we can reach every one of the \(2^{k}\) binary vectors in \(\{0,1\}^k\),
and always inside the interval \([1,2^{k}]\).
This proves surjectivity; since domain and codomain have the same size,
\(\mathcal P_k\) is bijective.
\end{proof}
For example in the case $\alpha=5$ and $\beta=1$ 
all eight length-\(3\) parity vectors occur once, and only once, among the integers \(1\le n\le 8\):
\(n=1\) produces \((1,1,0)\);
\(n=2=1+2^{0}\) gives \((0,1,1)\);
\(n=3=1+2^{1}\) yields \((1,0,0)\);
\(n=4=1+2^{0}+2^{1}\) gives \((0,0,1)\);
\(n=5=1+2^{2}\) yields \((1,1,1)\);
\(n=6=1+2^{0}+2^{2}\) gives \((0,1,0)\);
\(n=7=1+2^{1}+2^{2}\) yields \((1,0,1)\);
and \(n=8=1+2^{0}+2^{1}+2^{2}\) gives the remaining vector \((0,0,0)\).
Thus the map \(\mathcal P_3\) attains each binary vector exactly once on the interval \([1,8]\).

\subsection{Energy formula} Now we can prove Theorem \ref{MT2}, item $(i)$, namely the energy formula \eqref{Energy_conservation}.

\begin{lemma}[Energy conservation]\label{lemma: energy conservation}
	Let $k \in \mathbb{N}$ be fixed. Then for all $n\in\mathbb N,$ %We have that 
	\be\label{cons_energy}
	\sum_{j = 0}^{2^{k}-1} \left( C^k(n+j+ 2^k)-C^k(n+j)  \right)= (1+\alpha)^k.
	 \ee
\end{lemma}

\begin{proof}
	Since $n+j+2^k \equiv n+j \pmod{2^k}$ we have
	\[
	\ba		
	& \sum_{j = 0}^{2^{k}-1} \left( C^k(n+j+2^k)-C^k(n+j)\right) \\
	&~{} = 	\sum_{j = 0}^{2^{k}-1} \left( \dfrac{a(j,k)(n+j+2^k)+b(j,k)}{2^k}-\dfrac{a(j,k)(n+j)+b(j,k)}{2^k}\right) \\
	&~{} = \sum_{j = 0}^{2^k-1} a(j,k)  = (\alpha+1)^k.
	\ea
	\]
The proof is complete.
\end{proof}

\begin{remark}
It is interesting to compare Lemma \ref{lemma: energy conservation} with Theorem 1.2 in \cite{Terras}, that describes the periodicity of the parity vectors. In some sense, \eqref{cons_energy} reflects a weak periodicity, described by the independence on $n$ of the identity, and the invariance by movements of multiples of $2^k$ in the formula.   
\end{remark}
The formula \eqref{cons_energy} has an interesting consequence. Since $j+ 2^k -j =2^k$,
\be\label{pseudo_virial}
\frac1{2^k}\sum_{j = n}^{n+(2^{k}-1)} \left( \frac{C^k(j+ 2^k)-C^k(j)}{2^k}  \right)= \left(\frac{1+\alpha}4\right)^k.
\ee
The left hand side of this identity can be interpreted as the averaged sum of discrete slope quotients of the Collatz iteration $C^k$. The right hand side tell us that this quantity tends to grow with $k$ if $\alpha>3$, and stays bounded in the case $\alpha\leq 3.$ This result is in agreement with Proposition \ref{prop3p7} below.

Now we prove item $(ii)$, estimates \eqref{liminfsup} and \eqref{liminfsup1} in Theorem \ref{MT2}.

\begin{proposition}
	Let $n \in \mathbb{N}$ fixed, and consider $s_k$ and $s_{k,m}$ defined in \eqref{s_k}. Then
	\be\label{lim_n}
	\lim_{n\to +\infty} \dfrac{s_{k,m}(n)}{s_k(n)} =1,
	\ee
	and
	\be\label{infsup}
		\limsup_{k \to \infty} \dfrac{s_{k,m}(n)}{s_k(n)} \leq 1+2\alpha m, \quad \liminf_{k \to \infty} \dfrac{s_{k,m}(n)}{s_k(n)} \geq \begin{cases}  1+\dfrac{2m}{\alpha} \vspace{0.1cm} & \alpha \geq 3\\ 1+\dfrac{2m}{2\beta+1} & \alpha =1. \end{cases}
	\ee
\end{proposition}

\begin{proof} 
	Recall that 
	\[%be\label{s_k} 
	s_k(n) = \sum_{j = 0}^{2^k-1} C^k(n+j), \q s_{k,m}(n) = \sum_{j = 0}^{2^k-1}C^k(n+j+m2^k).
	\]%ee
	For the sake of simplicity, now we avoid the explicit dependence on $n$ in both sums above. First we need to estimate $s_k$. Notice that from \eqref{C_alpha} we have that for any $n,k \in \mathbb{N}$, $C^{k+1}(n)  \geq \dfrac{C^k(n)}{2}$, and using Lemma \ref{lemma: energy conservation} we obtain
	\[
	\ba
	\sum_{j = 0}^{2^{k+1}-1}C^{k+1}(n+j) \geq&~{}  \dfrac{1}{2}\sum_{j = 0}^{2^{k+1}-1}C^{k}(n+j) \\
	= &~{} \dfrac{1}{2}\sum_{j = 0}^{2^{k}-1}C^{k}(n+j)+\dfrac{1}{2}\sum_{j = 0}^{2^{k}-1}C^{k}(n+j+2^k) \\
	= &~{} \dfrac{1}{2}\sum_{j = 0}^{2^{k}-1}C^{k}(n+j)+\dfrac{1}{2}\sum_{j = 0}^{2^{k}-1}C^{k}(n+j)+\dfrac{(1+\alpha)^k}{2} \\
	= &~{} \sum_{j = 0}^{2^{k}-1}C^{k}(n+j)+\dfrac{(1+\alpha)^k}{2}.
	\ea
	\]
	We then see that $s_{k+1} - s_k \geq \frac{(1+\alpha)^k}{2}$, for all $k=0,1,\ldots$ Adding this inequality from $0$ to a fixed $k$ we obtain 
	\be\label{inf}
	\frac1{s_{k+1}} \leq \frac1{n+\dfrac{1}{2\alpha}\left((1+\alpha)^{k+1}-1\right)}, %s_{k+1} \geq s_0+\dfrac{1}{2\alpha}\left((1+\alpha)^{k+1}-1\right).
	\ee
	where we have used that $s_0=n$.
	
	To estimate $s_{k}$ from above we use from \eqref{C_alpha} and $\alpha\geq1$ that $C^{k+1}(n) \leq \dfrac{\alpha}{2}C^k(n)+\dfrac{\beta}{2}$. We proceed in the same way as before:
	\begin{align*}
		\sum_{j = 0}^{2^{k+1}-1}C^{k+1}(n+j) & \leq \frac{\alpha}{2}\sum_{j = 0}^{2^{k+1}-1}C^k(n+j)+2^k\beta \\
		& = \dfrac{\alpha}{2}\sum_{j = 0}^{2^{k}-1}C^{k}(n+j)+\dfrac{\alpha}{2}\sum_{j = 0}^{2^{k}-1}C^{k}(n+j+2^k)+2^k\beta \\
		& = \alpha \sum_{j = 0}^{2^k-1}C^k(n+j)+\frac{\alpha}{2}(1+\alpha)^k+2^k\beta.
	\end{align*}
	Thus $s_k$ must satisfy the inequality 
	\[ 
	s_{k+1}-\alpha s_k \leq \frac{\alpha}{2}(1+\alpha)^k+2^k\beta,
	\]
	or equivalently 
	\[ 
	\frac{s_{k+1}}{\alpha^{k+1}}-\frac{s_k}{\alpha^k} \leq \frac{1}{2}\left(\frac{1+\alpha}{\alpha}\right)^k+\left(\frac{2}{\alpha}\right)^k\frac{\beta}{\alpha}.
	\]
	Adding from 0 to $k$ and reordering the resulting expression we obtain 
	\[
	s_{k+1} \leq \left(s_0-\frac{\alpha}{2}+\frac{\beta}{\alpha-2}\right)\alpha^{k+1}+\frac{\alpha}{2}(1+\alpha)^{k+1}-\frac{\beta}{\alpha-2}2^{k+1},
	\]
	or, using that $s_0=n$,
	\be\label{sup}
	\frac1{s_{k+1}} \geq \frac1{\left(n -\frac{\alpha}{2}+\frac{\beta}{\alpha-2}\right)\alpha^{k+1}+\frac{\alpha}{2}(1+\alpha)^{k+1}-\frac{\beta}{\alpha-2}2^{k+1}}.
	\ee
	Having estimated the term $s_k$, we note that applying Lemma \ref{lemma: energy conservation} we have 
	\[ 
	s_{k,m} = s_{k,m-1}+ (1+\alpha)^k= \cdots = s_k+m(1+\alpha)^k,
	\]
	and then
	\be\label{cuota}
	\frac{s_{k,m}}{s_k} = 1+m\frac{(1+\alpha)^k}{s_k}.
	\ee
	Using the estimates \eqref{inf} and \eqref{sup} proved above and taking sup and inf limit we conclude the proof of \eqref{lim_n} and \eqref{infsup}, depending on $\alpha=1$ and $\alpha\geq 3$. Indeed, from \eqref{cuota} and \eqref{inf},
	\[
	 \frac{s_{k,m}}{s_k}  %= 1+m\frac{(1+\alpha)^k}{s_k} 
	 \leq 1+ \frac{m (1+\alpha)^k}{n+\dfrac{1}{2\alpha}\left((1+\alpha)^{k}-1\right)}.
	\]
	From \eqref{cuota} and \eqref{sup},
	\[
	\frac{s_{k,m}}{s_k} = 1+m\frac{(1+\alpha)^k}{s_k} \geq 1+ \frac{m(1+\alpha)^k}{\left(n -\frac{\alpha}{2}+\frac{\beta}{\alpha-2}\right)\alpha^{k}+\frac{\alpha}{2}(1+\alpha)^{k}-\frac{\beta}{\alpha-2}2^{k}}.
	\]
	Passing to the limit in $n$ we get \eqref{lim_n}. Taking liminf and limsup in $k$ we get \eqref{infsup}.
\end{proof}

Combining the energy identity and the previous result we can have some interesting consequences. Indeed, from the last estimates in the proof above
\[
 \frac{s_k}{\left(n -\frac{1}{2}\right)3^{k}+\frac{3}{2}4^{k}- 2^{k}}   \leq \frac1{4^k} s_k \left( \frac{s_{k,1}}{s_k} -1 \right)= 1 \leq  \frac{s_k }{n+\dfrac{1}{6}\left(4^{k}-1\right)}.
\]
This implies the inequalities 
\[
\limsup_{k\to+\infty}  \limsup_{n\to+\infty} \frac{s_k(n)}{3^k n} \leq 1 \quad \hbox{ and } \quad \liminf_{k\to+\infty}  \liminf_{n\to+\infty} \frac{s_k(n)}{ n} \geq 1.
\]

\subsection{Discrete time derivative} 

Let us consider now the discrete derivative of the Collatz orbit as defined in \eqref{DiscreteD}: 
\[
D_t(C^k(n)) = C^{k+1}(n)-C^k(n) = (C-I)(C^k(n)).
\]
Thus, the $m$-th derivative of Collatz orbit at time $k$ is 
\begin{equation}\label{derivada_m}
	D_t^m(C^k(n)) = (C-I)^m(C^k(m)) = \sum_{j = 0}^m \binom{m}{j}(-1)^jC^{m-j}(C^k(n)).
\end{equation}
If we examine the above identity, we note that the first term in the sum is $C^m(C^k(n))$, which, due to Lemma \ref{deco_coef} (applied with $k$ replaced by $m$), can be written as an affine function of $C^k(n)$ depending on the congruence class modulo $2^m$ of $C^k(n)$. The second term in the sum is $mC^{m-1}(C^k(n))$, and again, can be written as an affine function of $C^{k}(n)$, this time depending on the congruence class modulo $2^{m-1}$ of $C^{k}(n)$, but knowing the congruence class modulo $2^{m}$ give us the congruence class modulo $2^{m-1}$. Making the same analysis with all the rest of the terms in the sum reveals us that $D_t^m(C^k(n))$ can be written as an affine function of $C^k(n)$ depending only on its congruence class modulo $2^m$. 

To see an example, we note that for the usual Collatz map, that is, with $\alpha = 3, \beta = 1$, we have $C^2(n)$ can be written as \[ 
C^2(n) = \begin{cases}
	\frac{n}{4} & \text{if }n \equiv 0 \pmod 4 \vspace{0.1cm} \\
	\frac{3n+1}{4} & \text{if }n \equiv 1 \pmod 4 \vspace{0.1cm} \\
	\frac{3n+2}{4} & \text{if }n \equiv 2 \pmod 4 \vspace{0.1cm} \\
	\frac{9n+5}{4} & \text{if }n \equiv 3 \pmod 4,
\end{cases}
\]and then, since $D_t^2(C^k(n)) = (C^2-2C+I)C^k(n)$ we have 
\be\label{example}
D_t^2(C^k(n)) = \begin{cases}
	%	\begin{aligned}
		\frac{C^k(n)}{4} -2\frac{C^k(n)}{2}+C^k(n) & \text{if } C^k(n) \equiv 0 \pmod 4 \vspace{0.1cm} \\
		\frac{3C^k(n)+1}{4} -2\frac{3C^k(n)+1}{2}+C^k(n)& \text{if } C^k(n) \equiv 1 \pmod 4 \vspace{0.1cm} \\
		\frac{3C^k(n)+2}{4}-2\frac{C^k(n)}{2}+C^k(n)  &\text{if } C^k(n) \equiv 2 \pmod 4 \vspace{0.1cm} \\
		\frac{9C^k(n)+5}{4} -2\frac{3C^k(n)+1}{2}+C^k(n)& \text{if } C^k(n) \equiv 3 \pmod 4.
		%	\end{aligned}
\end{cases}
\ee
We want to make two remarks about the previous equation. The first one is that contribution of the term $-2C(C^k(n))$ appears repeated twice in the intermediate terms in \eqref{example} (as $-2\frac{C^k(n)}{2}$ and as $-2\frac{3C^k(n)+1}{2}$). This is related to the fact that if we consider all the 4 classes of equivalence modulo 4 we can see inside them 2 copies of the classes of equivalence modulo 2. The second one is that, in order to write each case as an affine map of $C^k(n)$ as is done in Lemma \ref{deco_coef}, the coefficients associated to the term $-2C(C^k(n))$ must to be multiplied by 2, in order to have a common denominator for all the terms. 

With this in mind, we have the following result, which ends the proof of Theorem \ref{MT2} $(iii)$.

\begin{proposition}\label{prop3p7}
	Let $m,n,k \in \NN$. Then $D_t^m(C^k(n))$ can be written as an affine function of $C^k(n)$ depending only on its congruence class modulo $2^m$. Moreover, the sum of the coefficients of $D_t^m(C^k(n))$ is $(\alpha-3)^m$ for whose are associated with $C^k(n)$ and $\beta(\alpha-3)^{m-1}$ for the free ones.
\end{proposition}

\begin{proof}
	First, we replace $C^k(n)$ by $n$, because we are only interested in it congruence modulo $2^m$.  By \eqref{derivada_m},  
	\[
	\ba
	D_t^m(n) = &~{} \sum_{j = 0}^m \binom{m}{j}(-1)^jC^{m-j}(n).% = D_{t,1}^m(n) n +D_{t,0}^m(n). 
	\ea
	\]
	For $k \in [0,2^m-1]\cap \NN$ and $j \leq m$ we define $k_j \in [0,2^j-1]\cap \NN$ as the unique integer such that $k \equiv k_j \pmod{2^j}$. Note that as $k$ runs from 0 to $2^m-1$, $k_j$ covers all the values between 0 and $2^j-1$, repeating each one $2^{m-j}$ times. In this way, using Lemma \ref{deco_coef}, for $n \equiv k \pmod{2^m}$ we can write $C^{m-j}(n)$ as 
	\begin{align*} 
		C^{m-j}(n) &= \frac{1}{2^{m-j}} \left( a(k_{m-j}, m-j)n+b(k_{m-j},m-j)\right) \\
		&= \frac{2^j}{2^{m}} a(k_{m-j}, m-j) n+ \frac{2^j}{2^m}b(k_{m-j},m-j).
	\end{align*}
	With this, for $n \equiv k \pmod{2^m}$, we can write \eqref{derivada_m} as 
	\[
	\begin{aligned} 
		D_t^m(n) &= \frac{1}{2^m}\sum_{j = 0}^{m}\binom{m}{j}(-1)^j2^ja(k_{m-j},m-j)n \\
		&~{} + \frac{1}{2^m}\sum_{j = 0}^{m}\binom{m}{j}(-1)^j2^jb(k_{m-j},m-j).
	\end{aligned}
	\]
	Thus the coefficient associated with $n$ in the decomposition of $D_t^m(n)$ is \[ 
	\sum_{j = 0}^{m}\binom{m}{j}(-2)^ja(k_{m-j},m-j),
	\]
	and the free one is \[ 
	\sum_{j = 0}^{m}\binom{m}{j}(-2)^jb(k_{m-j},m-j).
	\]Now we compute the sum of those coefficients. Recall that we need to sum over the $2^m$ classes of equivalence of $n$, that is 
	\begin{align*}
		\sum_{k = 0}^{2^m-1}\sum_{j = 0}^{m}\binom{m}{j}(-2)^ja(k_{m-j},m-j) &= \sum_{j = 0}^{m}\sum_{k = 0}^{2^m -1}\binom{m}{j}(-2)^ja(k_{m-j},m-j) \\
		&= \sum_{j = 0}^{m}\binom{m}{j}(-2)^j\sum_{k = 0}^{2^m -1}a(k_{m-j},m-j).
	\end{align*}
	Thanks to Lemma \ref{deco_coef} \eqref{sum} we have 
	\[ 
	\sum_{k = 0}^{2^{m-j}-1}a(k,m-j) = (1+\alpha)^{m-j},
	\]
	and since $k_j$ repeats all the values in $[0,2^{m-j}]\cap \NN$ $2^j$ times we have 
	\[
	\ba
	%\begin{align*}
	& \sum_{k = 0}^{2^m-1}\sum_{j = 0}^{m}\binom{m}{j}(-2)^ja(k_{m-j},m-j)= \sum_{j = 0}^{m}\binom{m}{j}(-2)^j2^j(1+\alpha)^{m-j}  \\
	&\quad =2^m \sum_{j = 0}^{m}\binom{m}{j}(-2)^j\left(\frac{1+\alpha}{2}\right)^{m-j} = 2^m\left(\frac{\alpha-3}{2}\right)^m = (\alpha-3)^m.
	%\end{align*}
	\ea
	\]
	For the free coefficients we proceed in the same way as above, using the binomial sum, and using again \eqref{sum} and separating cases. The proof of Theorem \ref{MT2} is complete.
\end{proof}

\end{document}